\documentclass[preprint,1p]{elsarticle}

\usepackage{amsmath,amsthm,amsfonts}
\usepackage{hhline}
\usepackage{bm}

\newtheorem{theorem}{Theorem}

\newcommand{\grad}{\mathop{\rm grad}\nolimits}
\renewcommand{\div}{\mathop{\rm div}\nolimits}

\journal{arXiv.org}

\begin{document}

\begin{frontmatter}

\title{Splitting schemes for poroelasticity and thermoelasticity problems\tnoteref{label1}}
\tnotetext[label1]{Supported by CJSC "OptoGaN" (contract N02.G25.31.0090); RFBR (project N13-01-00719A)}

\author[ykt]{Kolesov A.E.}
\ead{kolesov.svfu@gmail.com}

\author[msc]{Vabishchevich P.N.\corref{cor1}}
\ead{vabishchevich@gmail.com}

\author[ykt]{Vasilyeva M.V.}
\ead{vasilyevadotmdotv@gmail.com}

\cortext[cor1]{Corresponding author}

\address[ykt]{North-Eastern Federal University, Yakutsk, Russia}
\address[msc]{Nuclear Safety Institute, RAS, Moscow, Russia}

\begin{abstract}
In this work, we consider the coupled systems of linear unsteady partial differential equations, which arise in the modeling of poroelasticity processes. 
Stability estimates of weighted difference schemes for the coupled system of equations are presented.
Approximation in space is based on the finite element method.
We construct splitting schemes and give some numerical comparisons for typical poroelasticity problems. The results of numerical simulation of a 3D problem are presented. Special attention is given to using hight performance computing systems.
\end{abstract}


\end{frontmatter}

\section{Introduction}

Poroelasticity \cite{biot1941general, wang2000theory, minkoff2003coupled, Meirmanov} and thermoelasticity problems \cite{biot1956thermoelasticity, plas_lub, simocomputational, nowacki1975dynamic} are mathematically and physically analogous due to the fact that the pressure and temperature play a similar role in deformation of a body. 
For instance, a change in the temperature or a change in the pressure in a body results in equal normal strains in three orthogonal directions and no shear strains. 

The basic mathematical models of such problems include the Lame equation for the motion and the pressure/temperature  equations.
The fundamental point is that the system of equations is coupled: the equation for the motion includes the volume force, which is proportional to the temperature/pressure gradient, and the temperature/pressure equations include the term, which describes the compressibility of a medium.

To solve numerically the coupled quasi-stationary linear system of equations, we approximate our system using the finite element method \cite{hughes2012finite, fem_z1}. 
Variational formulations for poroelasticity and thermoelasticity problems and finite element approximations  are considered in \cite{haga2012causes, minkoff2006comparison, armero1999formulation}. 
Due to the incompressibility constraint on the displacement field at the initial state (incompressible elasticity and Stokes type problem), the finite element interpolation with equal order spaces for both the displacement and pressure/temperature fields is not correct. 
It is well known that in classical mixed formulations, the finite element spaces must satisfy the LBB stability conditions \cite{babuvska1973finite, brezzi1974existence, brezzi1991mixed}. 
These kind of discretizations for poroelasticity problems provides a lower order of convergence for the pressure in comparison with the displacements. 

Quasi-stationary problem (steady for motion and unsteady for the temperature/pressure) 
is solved using the weighted scheme \cite{LeVeque2007, Ascher2008}.
The stability analysis is performed \cite{gaspar2003finite, lisbona2001operator} in the framework of the general theory of stability for operator-difference schemes \cite{Samarskii1989,SamarskiiMatusVabischevich2002}. 

At present, different classes of additive operator--difference schemes for evolutionary equations are constructed via an additive representation of the main operator onto several terms. 
Additive schemes are constructed using  splitting; they are associated with transition to a new time level on the basis of the solution of simpler problems for individual operators in the additive decomposition \cite{Marchuk1990,vabAdd}. 
We consider splitting (additive) schemes for poroelasticity/thermoelasticity problems with additive representation of the operator at the time derivative; we also consider  modifications of splitting schemes.
Similar schemes are examined in many studies \cite{kim2010sequential, mikelic2013convergence, armero1992new, jha2007locally}, but in our work, we show that these schemes are no more than regularization schemes \cite{samarskii1967regularization, vabAdd}.

The work is organized as follows.
Section 2 provides the mathematical model for the  poroelasticity problem, which is the same as the thermoelasticity problem. 
In Section 3, we consider properties of the differential problem and give a priory estimates for the stability for the solutions with respect to initial the data and the right-hand side. In this section, we also formulate our problem as  an initial value problem for a system of linear ordinary differential equations. 
Discretizations in space and in time are performed in Section 4. Here we conduct some analysis of the weighted differential schemes for coupled system. 
The central part of the work deals with the construction and numerical comparison of the stability of the splitting schemes.
In Section 5, we consider the additive (splitting) schemes that are compared for typical poroelasticity problems in Section 7.
Some modifications for additive (splitting) schemes associated with regularization are introduced in Section 6.
The results of numerical simulation of a 3D problem on hight performance computing systems are presented in Section 8. 

\section{Problem formulation} 

The linear poroelasticity equations can be expressed  \cite{biot1941general} as
\begin{equation}
\label{eq:poroelas}
\begin{split}
 \div \bm \sigma (\bm u) - \alpha \grad (p)  & = 0, \\
 \alpha \frac{\partial \div \bm u}{\partial t} + S \frac{\partial p}{\partial t} - \div \left( \frac{k}{\nu} \grad p \right) & = f(\bm x, t), 
\end{split}
\end{equation}
with the boundary conditions
\begin{equation}
\label{eq:bc}
\begin{split}
\bm \sigma \bm n = 0, \quad \bm x \in \Gamma^u_N, \quad & 
\bm u = 0, \quad \bm x \in \Gamma^u_D, \\
-\frac{k}{\nu} \frac{\partial p}{\partial n}= 0, \quad \bm x \in \Gamma^p_N, \quad & 
p = p_1, \quad \bm x \in \Gamma^p_D, \\
\end{split}
\end{equation}
and the initial conditions 
\begin{equation}
\label{eq:ic}
p(\bm x, 0) = p_0, \quad \bm x \in \Omega.
\end{equation}
Here the primary variables are  the fluid pressure $p$ and  the displacement  vector $\bm u$. Also, $\bm \sigma$ is the stress tensor,
 $S= 1/M$,  $M$ is the Biot modulus, $k$ is the permeability, $\nu$ is the fluid viscosity, $\alpha$ is the Biot-Willis fluid/solid coupling coefficient  and $f$ is a source term representing injection or production processes. Body forces are neglected, $\bm n$ is the unit normal to the boundary.
 
The stress tensor is given by
\begin{equation}
\label{eq:stress}
\bm \sigma = 2 \mu \bm \varepsilon(\bm u) + \lambda \div(\bm u) \, \bm I,
\end{equation}
where $\bm \varepsilon$ is the strain tensor:
\[
\bm \varepsilon(\bm u) = \frac{1}{2} \left( \grad \bm u + \grad \bm u^T \right),
\]
and $\mu$, $\lambda$ are Lame coefficients, $\bm I$ is the identity tensor.

In the case of thermoelasticity, the governing equations are the same as equations (\ref{eq:poroelas}) \cite{biot1956thermoelasticity, plas_lub, simocomputational, nowacki1975dynamic} with the temperature $T$ instead of the pressure:
\begin{equation}
\label{eq:thermo}
\begin{split}
 \div \bm \sigma(\bm u) -\beta \grad (T) & = 0,\\
 \beta T_0 \frac{\partial \div \bm u}{\partial t} + c \frac{\partial T}{\partial t}  - \div \left(\kappa \grad T \right) & = 0, 
\end{split}
\end{equation}
where $c$ is the heat capacity of the unit volume in the absence of deformation, $\kappa$ is the thermal conductivity and $\beta$ is the  coupling coefficient playing the similar role as the Biot-Willis coefficient $\alpha$. Here $T_0$ is the initial temperature of a medium.

\section{Differential problem properties}

We define the standard Hilbert space $H=L_2\left(\Omega\right)$ for the pressure with the following inner product and norm:
\[
(u, v) = \int_{\Omega} u(\bm x) \, v(\bm x) \, dx, \quad ||u|| = (u, u)^{1/2}
\]
and the Hilbert space $\bm H= \left( L_2 \left(\Omega \right) \right)^d$ for the displacement.  
Here $d = 2,3$ is the number of spatial dimensions. 

In $H$, we consider (see, e.g., $\mathcal{A}$ \cite{gaspar2003finite, lisbona2001operator, samarskii1978methods}) the operator
\begin{equation}
\label{eq:Au}
\mathcal{A} \bm v = -\mu \nabla^2 \bm v - (\lambda + \mu) \grad \div \bm v.
\end{equation}
The operator $\mathcal{A}$ is positive and self-adjoint in $\bm H$:
\[
(\mathcal{A} \bm v, \bm v) \geq 0, \quad (\mathcal{A} \bm v, \bm u) = (\bm v, \mathcal{A} \bm u).
\]

Similarly, in $H$, we define the operator $\mathcal{B}$ as follows:
\begin{equation}
\label{eq:Bu}
\mathcal{B} p = -\div \left( \frac{k}{\eta} \grad p \right).
\end{equation}
The coupling terms in the poroelasticity problem are associated with the gradient and divergence operators denoted by $\mathcal{G}$ and $\mathcal{D}$:
\begin{equation}
\label{eq:grad_div}
\left( \mathcal{G} p, \bm u \right) = -\left( \mathcal{D} \bm u, p \right).
\end{equation}

Now equations (\ref{eq:poroelas}) may be written in the operator-differential  form as an abstract initial value problem:
\begin{equation}
\label{eq:mmd}
\begin{split}
\mathcal{A} \bm u + \alpha \mathcal{G}p  & = 0,  \\
\frac{d}{dt} \left( S \, p + \alpha \mathcal{D} \bm u \right) + \mathcal{B} p &= f(t),
\quad t > 0, 
\end{split}
\end{equation}
with the initial condition for the pressure:
 \[
 \quad p(0) = p_0.
 \]
Let us define the inner products associated with $\mathcal{A}$ and $\mathcal{B}$:
\[
(u, v)_{\mathcal{A}} = \left( \mathcal{A} u, v\right), \quad (u, v)_{\mathcal{B}} = \left( \mathcal{B} u, v\right)
\]
and the norms: 
\[
||u||_{\mathcal{A}} = (\mathcal{A} u, u)^{1/2}, \quad 
||u||_{\mathcal{B}} = (\mathcal{B} u, u)^{1/2}.
\]

\begin{theorem}
\label{thm:apriopi_dfc}
The solution of problem (\ref{eq:mmd})  satisfies the a priory estimate
\begin{equation}
\label{eq:esOp}
||\bm u(t)||^2_{\mathcal{A}} + S||p(t)||^2 \leq ||\bm u(0)||^2_{\mathcal{A}} + S||p(0)||^2 + \frac{1}{2}\int^t_0 ||f(s)||^2_{\mathcal{B}^{-1}} ds.
\end{equation}
\end{theorem}

\begin{proof}
Multiplying the first and the second equations of (\ref{eq:mmd})  by $d \bm u /dt$ and $p$, respectively, we get
\[
\left( \mathcal{A} \bm u, \frac{d \bm u}{dt} \right) + \alpha \left( \mathcal{G} p, \frac{d \bm u}{dt} \right) = 0,
\]
\[
\frac{d}{dt} \left( S( p, p) + \alpha \left( \mathcal{D} \bm u, p \right) \right) +  \left( \mathcal{B} p, p \right) = (f, p).
\]
Adding this equations and using the relation (\ref{eq:grad_div}), we obtain
\[
\left( \mathcal{A} \bm u, \frac{d \bm u}{dt} \right) + S\left(p, \frac{dp}{dt} \right) +  \left( \mathcal{B} p, p \right) = (f, p).
\]
Using the Cauchy-Schwarz inequality for the right-hand side:
\[
(f, p) \leq ||p||^2_{\mathcal{B}} + \frac{1}{4} ||f||^2_{\mathcal{B}^{-1}},
\]
the following inequality is obtained:
\[
\left( \mathcal{A} \bm u, \frac{d \bm u}{dt} \right) +  S\left(p, \frac{dp}{dt} \right)  \leq  \frac{1}{4} ||f||^2_{\mathcal{B}^{-1}}.
\]
In view of 
\[
\left( \frac{d u}{dt}, u \right) = \frac{1}{2} \frac{d}{dt} (u, u),
\]
we have
\[
\frac{d}{dt} \left( 
(\mathcal{A} \bm u, \bm u) + S(p, p)
\right)
\leq  \frac{1}{2} ||f||^2_{\mathcal{B}^{-1}}.
\]
Time integration gives  the a priory estimate (\ref{eq:esOp}) providing stability with respect to the initial data and the right-hand side.
\end{proof}

After differentiation with respect to time of the displacement equation, the operator-differential  form of poroelasticity problem (\ref{eq:mmd}) may be written in a more convenient form:
\begin{equation}
\label{eq:mm2}
\begin{split}
\mathcal{A} \frac{d \bm u}{d t} + \alpha \mathcal{G} \frac{d p}{d t} & = 0,  \\
\frac{d}{dt} \left( S \, p + \alpha \mathcal{D} \bm u \right) + \mathcal{B} p & = f(t),
\quad t > o, 
\end{split}
\end{equation}
which is an initial value problem for a system of linear ordinary differential equations. 

Let $\bm U = \{\bm u, p\}$ be the vector of unknowns, $\bm F = \{\bm 0, f\}$ be the given vector of the right-hand sides and
\begin{equation}
\label{eq:ODE}
\mathbb{B} \frac{d \bm U}{dt} + \mathbb{A} \bm U = \bm F, \quad t > 0,
\end{equation}
with the initial conditions
\[
\bm U = \bm U_0.
\]
Here 
\begin{equation}
\label{eq:odeB1}
\mathbb{B} = 
\begin{pmatrix}
\mathcal{A} & \alpha \mathcal{G}\\ 
\alpha \mathcal{D} & S \mathcal{I}
\end{pmatrix}, 
\quad
\mathbb{A} = 
\begin{pmatrix}
0 & 0\\ 
0 & \mathcal{B}
\end{pmatrix}.
\end{equation}

Multiplying equation (\ref{eq:ODE}) by $\bm U$, we obtain
\[
\frac{1}{2} \frac{d}{dt} \left( \mathbb{B} \bm U, \bm U \right) + \left( \mathbb{A} \bm U, \bm U \right)  = \left( \bm F, \bm U \right).
\]
For the right-hand side, we use the estimate
\[
\left( \bm F, \bm U \right) \leq  \left( \mathbb{A} \bm U, \bm U \right) + \frac{1}{4}  \left( \mathbb{A}^{-1} \bm F, \bm F \right).
\]
This provides the a priory estimate for the solution:
\begin{equation}
\label{eq:es2}
||\bm U(t)||_{\mathbb{B}}^2 \leq ||\bm U_0||_{\mathbb{B}}^2 + \frac{1}{2} \int^t_0 ||\bm F(s)||^2_{\mathbb{A}^{-1}} ds.
\end{equation}
Note that this estimate is similar to the a priory estimate (\ref{eq:esOp}), but it is obtained in a more simple way.

\section{Discretization in space and in time}

For the numerical solution of the problem, first, we come to a variational problem by multiplying the first and  
the second equations of (\ref{eq:poroelas}) by test functions $\bm v$ and $q$, respectively, and  
integrating by parts to eliminate the second order derivatives. 
Find $\bm u \in \bm V$, $p \in Q$ such that
\begin{equation}
\label{eq:var}
\begin{split}
\int_{\Omega} \bm \sigma(\bm u) \, \bm \varepsilon(\bm v) dx  
& +  \int_{\Omega}\alpha (\grad p, \bm v) dx = 0, \quad \forall \bm v \in \hat{ \bm V} ,\\
\int_{\Omega} \alpha  \frac{d \div \bm u}{dt} q dx  & + 
\int_{\Omega} S \frac{d p}{dt} q dx + 
\int_{\Omega}  \left( \frac{k}{\nu} \grad p, \grad q\right) dx \\
& = \int_{\Omega} f \, q \,dx, \quad \forall q\in \hat{ Q}.
\end{split}
\end{equation}
Here the spaces of trial and test functions are defined as
\begin{align*}
  \bm V  & = \lbrace \bm v \in [H^1(\Omega)]^d: \bm v(\bm x) = 0, \bm x \in \Gamma_D^u  \rbrace, \\
  \hat{\bm V} & = \lbrace \bm v \in [H^1(\Omega)]^d: \bm v(\bm x) = 0, \bm x \in \partial \Omega \rbrace, \\
  Q  & = \lbrace  q \in H^1(\Omega):  q(\bm x) = p_1, \bm x \in \Gamma_D^p \rbrace, \\
  \hat{Q} & = \lbrace  q \in H^1(\Omega): q(\bm x) = 0, \bm x \in \partial \Omega \rbrace,
\end{align*}
where  $H^1$ is a Sobolev space.

A discrete problem is obtained by restricting the variational problem (\ref{eq:var}) to pairs of discrete  spaces of trial and test functions \cite{haga2012causes, logg2012automated}: find $\bm u_h \in \bm V_h \subset \bm V$, $p_h \in Q_h \subset Q$ such that  
\begin{equation}
\label{eq:fem}
\begin{split}
\int_{\Omega} \bm \sigma(\bm u_h) \, \bm \varepsilon(\bm v_h) dx  
& +  \int_{\Omega}\alpha (\grad p_h,  \bm v_h) dx = 0, \quad \forall \bm v_h \in \hat{ \bm V_h} \subset \hat{ \bm V} ,\\
\int_{\Omega} \alpha  \frac{d \div \bm u_h}{dt} q_h dx  & + 
\int_{\Omega} S \frac{d p_h}{dt} q_h dx + 
\int_{\Omega}   \left( \frac{k}{\nu} \grad p_h,\grad q_h \right) dx  \\ & = 
\int_{\Omega} f \, q_h \,dx, \qquad \quad \forall q_h \in \hat{ Q_h} \subset \hat{ Q}.
\end{split}
\end{equation}

For discretization in time, we employ the weighted difference scheme for time-stepping.
We define the following bilinear and linear forms over the domain $\Omega$:
\[
\begin{split}
a(\bm u, \bm v)  &= \int_{\Omega} \bm \sigma(\bm u) \, \bm \varepsilon(\bm v) dx, \\
g(p, \bm v)  &= \int_{\Omega}\alpha (\grad p, \bm v) dx,
\\
c(p, q) &= \int_{\Omega} S \, p \, q \, dx, 
\\
d(\bm u, q) &= \int_{\Omega} \alpha \div \bm u \, q \, dx , 
\\
b(p, q) &= \int_{\Omega} \left( \frac{k}{\nu} \grad p, \grad q \right) dx  , 
\\
l(f, q) &= \int_{\Omega} f  \, q \, dx,
\end{split}
\]
Note that for all $u$ and $v$, we have
\[
d(u, v) = -g(v, u)
\]

Let $\bm u^n = \bm u(\bm x, t_n)$,  $p^n = p(\bm x, t_n)$, where $t_n = n \tau$, $n = 0, 1, ..., $ and $\tau > 0$.
Then the problem is reformulated as follows: we search $p \in Q$ and $\bm u \in \bm V$ that satisfy the following relations:

for the displacement
\begin{equation}
\label{eq:canu}
a(\bm u^{n+1}, \bm v) + g(p^{n+1}, \bm v) = 0, \quad \forall \bm v \in \hat{\bm V}, 
\end{equation}

for the pressure
\begin{equation}
\label{eq:canp}
\begin{split}
   d \left( \frac{\bm u^{n+1} - \bm u^n}{\tau}, q \right)  
+ &c \left( \frac{p^{n+1} - p^n}{\tau}, q \right) \\
+ &b(p_{\theta}^{n+1} , q) = l(f_{\theta}^{n+1} , q), \quad \forall q \in \hat Q,
\end{split}
\end{equation}
with 
\[
p_{\theta}^{n+1} = \theta p^{n+1}+(1-\theta) p^n,
\]
\[
f_{\theta}^{n+1} = f(\bm x, \theta t^{n+1}+(1-\theta) t^n)
\]
and $0 \leq \theta \leq 1$.

\begin{theorem}
For $\theta \geq 0.5$, the solution of the weighted difference scheme (\ref{eq:canu})-(\ref{eq:canp}) satisfies the a priory estimate
\begin{equation}
\label{eq:es}
||\bm u^{n+1}||_a^2 + ||p^{n+1}||_c^2 \leq  ||\bm u^{n}||_a^2 + ||p^{n}||_c^2 + \frac{\tau}{2} ||f_{\theta}^{n+1} ||_{*, b}^2
\end{equation}
\end{theorem}

\begin{proof}
Let 
\[
\bm u_{\theta}^{n+1} = \theta \bm u^{n+1}+(1-\theta) \bm u^n,
\]
\[
v = \frac{\bm u^{n+1} - \bm u^n}{\tau}
\]
in  the equation for the displacement (\ref{eq:canu}), then we have 
\begin{equation}
\label{eq:proof-canu}
a \left(\bm u_{\theta}^{n+1}, \frac{\bm u^{n+1} - \bm u^n}{\tau}\right) + g \left(p_{\theta}^{n+1}, \frac{\bm u^{n+1} - \bm u^n}{\tau}\right) = 0.
\end{equation}
By analogy, substituting $q = p_{\theta}^{n+1}$ in the pressure equation (\ref{eq:canp}), we obtain
\begin{equation}
\label{eq:proof-canp}
\begin{split}
    c \left( \frac{p^{n+1} - p^n}{\tau},  p_{\theta}^{n+1} \right) 
+ &d \left( \frac{\bm u^{n+1} - \bm u^n}{\tau},  p_{\theta}^{n+1} \right) \\
+ &b(p_{\theta}^{n+1} ,  p_{\theta}^{n+1}) = (f_{\theta}^{n+1} ,  p_{\theta}^{n+1}).
\end{split}
\end{equation}

Adding (\ref{eq:proof-canu}) and (\ref{eq:proof-canp}), we get
\[
    a \left(\bm u_{\theta}^{n+1}, \frac{\bm u^{n+1} - \bm u^n}{\tau} \right)
+ c \left( \frac{p^{n+1} - p^n}{\tau},  p_{\theta}^{n+1} \right) 
+ b(p_{\theta}^{n+1} ,  p_{\theta}^{n+1}) = (f_{\theta}^{n+1} ,  p_{\theta}^{n+1}).
\]
Taking into account the inequality 
\[
( f_{\theta}^{n+1}, p_{\theta}^{n+1}) \leq \varepsilon ||p_{\theta}^{n+1}||_b^2 + \frac{1}{4} ||f_{\theta}^{n+1}||_{*, b}^2
\]
then for $\varepsilon = 1$
\[
    a \left(\bm u_{\theta}^{n+1}, \frac{\bm u^{n+1} - \bm u^n}{\tau} \right)
+ c \left( \frac{p^{n+1} - p^n}{\tau},  p_{\theta}^{n+1} \right) 
\leq \frac{1}{4} ||f_{\theta}^{n+1}||_{*, b}^2 .
\]
Here the right-hand side is estimated in the norm for the space adjoint to $H_b$  using symbols $||v||_{*, b}$.

Using the identity 
\[
  v^{n+1}_{\theta} = \theta v^{n+1} + (1-\theta) v^n = \frac{ v^{n+1} +  v^n}{2} + \left(\theta -\frac{1}{2}\right) ( v^{n+1} -  v^n ),
\]
and employing that for symmetric bilinear form $d(u,  v) = d(v,  u)$ it is satisfied:
\[
d(u+v, u-v) = d(u, u) - d(v, v),
\]
we obtain the inequality
\[
\begin{split}
\frac{1}{2} \left( ||u^{n+1}||^2_a - ||u^n||^2_a \right) & + 
\frac{1}{2} \left( ||p^{n+1}||^2_c - ||p^n||^2_c \right) \\
& + \left(\theta -\frac{1}{2}\right) \left( ||u^{n+1} - u^n||^2_a  + ||p^{n+1} - p^n||^2_c \right)
\leq \frac{\tau}{4} ||f_{\theta}^{n+1}||_{*, b}^2 .
\end{split}
\]
If $\theta \geq 0.5$, then the estimate (\ref{eq:es}) holds; this ensures stability with respect to the initial data and the right-hand side.
\end{proof}

\section{Additive schemes}

Let $H$ be a finite-dimensional Hilbert space, and $\mathbb{B}, \mathbb{A}$ are linear operators in $H$.
We introduce a space $H_\mathcal{D}$ with the inner product and norm:
\[
(y, w)_{\mathcal{D}} = \left( \mathcal{D} y, w \right), \quad ||y||_{\mathcal{D}} = \left( \mathcal{D} y, w \right)^{1/2}.
\]
We consider a initial value problem for a system of linear ordinary differential equations: find $\bm U(t) \in H$ such that
\begin{equation}
\label{eq:IVP}
\mathbb{B} \frac{d \bm U}{dt} + \mathbb{A} \bm U = \bm F, \quad t > 0,
\end{equation}
\begin{equation}
\label{eq:odeB}
\mathbb{B} = 
\begin{pmatrix}
\mathcal{A} & \alpha \mathcal{G}\\ 
\alpha \mathcal{D} & S \mathcal{I}
\end{pmatrix}, 
\quad
\mathbb{A} = 
\begin{pmatrix}
0 & 0\\ 
0 & \mathcal{B}
\end{pmatrix},
\end{equation}
with the initial conditions
\begin{equation}
\label{eq:icIVP}
\bm U = \bm U_0.
\end{equation}
For the problem (\ref{eq:IVP})-(\ref{eq:icIVP}), we have the a priory estimate (\ref{eq:es2}), which expresses the stability of the solution with respect to the initial data and the right-hand side.

In our problem, the computational complexity is associated with the operator $\mathbb{B}$ at the time derivative. In this case, to decrease the computational complexity of the problem (\ref{eq:IVP})-(\ref{eq:icIVP}), we employ the additive representation:
\begin{equation}
\label{eq:decB}
\mathbb{B} = \mathbb{B}_0 + \mathbb{B}_1,
\end{equation}
where we take $\mathbb{B}_0$ as an easily invertible operator; this can help us to decouple the problem. 

Splitting schemes for the approximate solution of (\ref{eq:IVP})-(\ref{eq:icIVP}) will be constructed on the basis of the weighted difference schemes.
The standard two-level weighted scheme for the problem (\ref{eq:IVP})-(\ref{eq:icIVP}) has the following form:
\begin{equation}
\label{eq:sc-IVP}
\mathbb{B} \frac{\bm U^{n+1} - \bm U^n}{\tau} + \mathbb{A} \left( \theta \bm U^{n+1} + (1-\theta)\bm U^n \right) = \bm F^n, \quad n=0, 1, ...
\end{equation}
where, for example,
\[
\bm F^n = \bm F(\theta t^{n+1} + (1-\theta)t^n),
\]
and $\theta$ is the weight parameter.

To solve the problem (\ref{eq:IVP})-(\ref{eq:icIVP}) with the additive operator $\mathbb{B}$, we apply the following difference scheme:
\begin{equation}
\label{eq:DLU}
\begin{split}
\mathbb{B}_0 \frac{\bm U^{n+1} - \bm U^n}{\tau} 
& + \mathbb{B}_1 \frac{\bm U^n - \bm U^{n-1}}{\tau}  \\
&+ \mathbb{A} 
\left( \theta_1 \bm U^{n+1} 
+ (1-\theta_1 - \theta_2)\bm U^n + \theta_2 \bm U^{n-1}  
\right) = \bm F^n, \\
 & \quad n=0, 1, ...
\end{split}
\end{equation}
Unlike (\ref{eq:sc-IVP}), the scheme (\ref{eq:DLU}) is a three-level scheme with two weight factors $\theta_1$ and $\theta_2$. 

As the operators $\mathbb{B}_0$, $\mathbb{B}_1$ in (\ref{eq:DLU}), we can take the following representations:
\begin{equation}
\label{eq:sc-D}
\mathbb{B}_0 = 
\begin{pmatrix}
\mathcal{A} & 0\\ 
0 & S \mathcal{I}
\end{pmatrix}, 
\quad
\mathbb{B}_1 = 
\begin{pmatrix}
0 & \alpha \mathcal{G}\\ 
\alpha \mathcal{D} & 0
\end{pmatrix}, 
\end{equation}
where the diagonal part of the operator $\mathbb{B}$ is separated. 
In terms of numerical implementation, it is convenient to use the triangular splitting. 
The first case has the form:  
\begin{equation}
\label{eq:sc-L}
\mathbb{B}_0 = 
\begin{pmatrix}
\mathcal{A} & 0\\ 
\alpha \mathcal{D} & S \mathcal{I}
\end{pmatrix}, 
\quad
\mathbb{B}_1 = 
\begin{pmatrix}
0 & \alpha \mathcal{G}\\ 
0 & 0
\end{pmatrix}, 
\end{equation}
and the second one is represented as: 
\begin{equation}
\label{eq:sc-U}
\mathbb{B}_0 = 
\begin{pmatrix}
\mathcal{A} & \alpha \mathcal{G}\\ 
0 & S \mathcal{I}
\end{pmatrix}, 
\quad
\mathbb{B}_1 = 
\begin{pmatrix}
0 & 0\\ 
\alpha \mathcal{D} & 0
\end{pmatrix} .
\end{equation}
In each of these decompositions, we have the part $\mathbb{B}_0$ that is easily invertible.

\section{Regularized schemes}

Now we consider some modified techniques, which often discussed in the literature \cite{kim2010sequential, mikelic2013convergence, armero1992new, jha2007locally}.
The undrained split method consists in imposing a constant fluid mass during the structure deformation \cite{kim2010sequential, mikelic2013convergence}. We set
\[
p^{n+1} = p^n - \frac{\alpha}{S} \mathcal{D} (\bm u^{n+1} - \bm u^n),
\]
then we substitute this into the displacement equation:
\[
\mathcal{A} \bm u^{n+1} + \alpha \mathcal{G} p^n - \tau \frac{\alpha^2}{S} \mathcal{G} \mathcal{D} \frac{\bm u^{n+1} - \bm u^n}{\tau} = 0 .
\]
The equations become as follows:
\begin{equation}
\label{eq:undr}
\begin{split}
\mathcal{A} \bm u^{n+1} + \alpha \mathcal{G} p^n - \tau \frac{\alpha^2}{S} \mathcal{G} \mathcal{D} \frac{\bm u^{n+1} - \bm u^n}{\tau} & = 0, \\
 S \frac{p^{n+1} - p^n}{\tau} + \alpha \mathcal{D} \frac{\bm u^{n+1} - \bm u^n}{\tau} + \mathcal{B} p^{n+1} & = f^n.
\end{split}
\end{equation}
The displacement equation in (\ref{eq:undr}) can be generalized as a regularized scheme
(see, e.g., \cite{vabAdd,samarskii1967regularization}):
\begin{equation}
\label{eq:RL}
\mathcal{A} \bm u^{n+1} + \alpha \mathcal{G} p^n + \beta \tau \mathcal{R} \frac{\bm u^{n+1} - \bm u^n}{\tau} = 0,
\end{equation}
where, for the undrained split, we have
\[
\mathcal{R}  = -\frac{\mathcal{G} \mathcal{D}}{S}, \quad \beta = \alpha^2.
\]

The fixed stress split method consist in imposing constant volumetric mean total stress \cite{kim2010sequential, mikelic2013convergence}. We set
\[
\mathcal{D} \bm u^{n+1} = \mathcal{D} \bm u^n +\alpha \frac{1}{K_{dr}} (p^{n+1} - p^n),
\]
where $K_{dr}$ is the constrained (drained) modulus:
\[
K_{dr} = \frac{E (1- \nu)}{(1-2 \nu) (1+\nu)}.
\]
Substitution into the pressure equation leads to
\[
 S \frac{p^{n+1} - p^n}{\tau} + \alpha \mathcal{D} \frac{\bm u^{n} - \bm u^{n-1}}{\tau}  +\left( \frac{\alpha^2}{K_{dr}} \frac{p^{n+1} - p^n}{\tau} - \frac{\alpha^2}{K_{dr}} \frac{p^n - p^{n-1}}{\tau} \right) + \mathcal{B} p^{n+1} = f^n,
\]
and now we get
\begin{equation}
\label{eq:fixstress}
\begin{split}
\mathcal{A} \bm u^{n+1} &+ \alpha \mathcal{G} p^{n+1} = 0, \\
 S \frac{p^{n+1} - p^n}{\tau} &+ \alpha \mathcal{D} \frac{\bm u^{n} - \bm u^{n-1}}{\tau}  \\
& +\left( \frac{\alpha^2}{K_{dr}} \frac{p^{n+1} - p^n}{\tau} - \frac{\alpha^2}{K_{dr}} \frac{p^n - p^{n-1}}{\tau} \right) + \mathcal{B} p^{n+1} = f^n.
\end{split}
\end{equation}

Similarly, we can generalize the pressure equation in (\ref{eq:fixstress}) as a regularized scheme \cite{vabAdd,samarskii1967regularization}:
\begin{equation}
\label{eq:RU}
 S \frac{p^{n+1} - p^n}{\tau} + \alpha \mathcal{D} \frac{\bm u^{n} - \bm u^{n-1}}{\tau} + \beta \tau \mathcal{R} \frac{p^{n+1} - 2 p^n + p^{n-1}}{\tau^2}  + \mathcal{B} p^{n+1} = f^n,
\end{equation}
where, for the fixed stress split, we have
\[
\mathcal{R}  = \frac{1}{K_{dr}}, \quad \beta = \alpha^2.
\]

\section{Numerical tests for poroelasticity problems}

Let us compare the additive scheme (\ref{eq:DLU}) with the additive operator 
\[
\mathbb{B} = \mathbb{B}_0 + \mathbb{B}_1,
\] 
and  the regularized schemes (\ref{eq:RL}), (\ref{eq:RU}) on the tests, where coefficients are typical for poroelasticity problems.
For numerical implementation, we use \texttt{Gmsh} \cite{geuzaine2009gmsh} for mesh generation and \texttt{Paraview} \cite{henderson2004paraview} for visualization of numerical results.
Our code is based on the library for scientific computations \texttt{FEniCS}
\cite{logg2012automated}. Two test cases have been predicted. 

Figure~\ref{pic:domain} presents the computational domain whereas Figure~\ref{pic:mesh} demonstrates two meshes  used for the numerically silving poroelasticity problem.
Parameters of problem are presented in Table~\ref{tab:2d-param} for Test~1 and Test~2. 
We use the following boundary conditions:
\[
\sigma_{x_1} = 0, \quad \sigma_{x_2} = 0, \quad \bm x \in \Gamma_1,
\]\[
u_{x_1} = 0, \quad \sigma_{x_2} = 0, \quad \bm x \in \Gamma_2,
\]\[
u_{x_1} = 0, \quad u_{x_2} = 0, \quad \bm x \in \Gamma_3.
\]
The time steps $\tau = 0.1$ day and $t_{max} = 3$ days were used on the finest spatial mesh with 40000 cells.  
The pressure field along with the displacement and stress (von Mises) distributions  are depicted in Fig.~\ref{pic:sol} as the upper, middle and lower isocontours, respectively, for the second test case.  

\begin{figure}
\begin{center}
\includegraphics[width=0.8\linewidth]{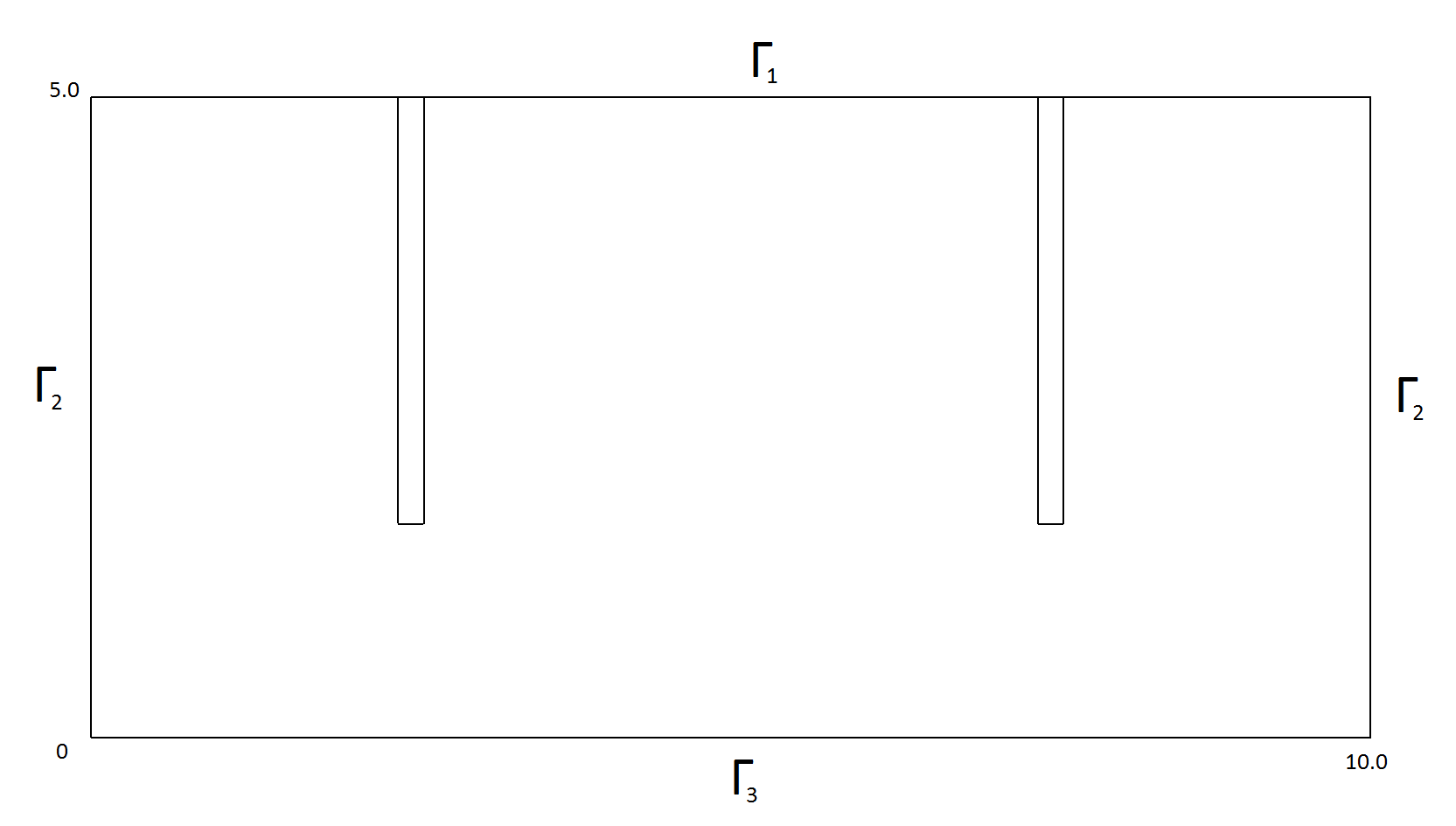}  
\end{center}
\caption{Computational domain}
\label{pic:domain} 
\end{figure}

\begin{table}
\begin{center}
\begin{tabular}{llll}
\hline 
Parameter  & Symbol & Test 1 & Test 2 \\ \hline 
Biot modulus					& $M$   			&  $5.0$  GPa				&  $50.0$  GPa\\
Shear 	modulus				& $\mu$			&  $5.0 $ GPa 				&  $15.0 $ GPa\\
Lame constant					& $\lambda$	&  $5.0$   GPa				&  $10.0$   GPa\\
Absolute permeability		&$k$				&  $10^{-17}$ m$^2$ &  $10^{-18}$ m$^2$\\
Fluid viscosity					& $\nu$			& $0.001$ Pa/sec  		& $0.001$ Pa/sec\\
Initial pressure					& $p_0$			&  $10.0$  MPa 			&  $10.0$  MPa\\
Injector pressure				& $p_i$			&  $10.01$ MPa 			&  $10.01$ MPa\\
Producer pressure			&$p_p$			&  $9.99$	MPa 				&  $9.99$	MPa\\
\hline 
\end{tabular} 
\end{center}
\caption{Problem properties}
\label{tab:2d-param}
\end{table}

\begin{figure}
\begin{center}
\includegraphics[width=0.9\linewidth]{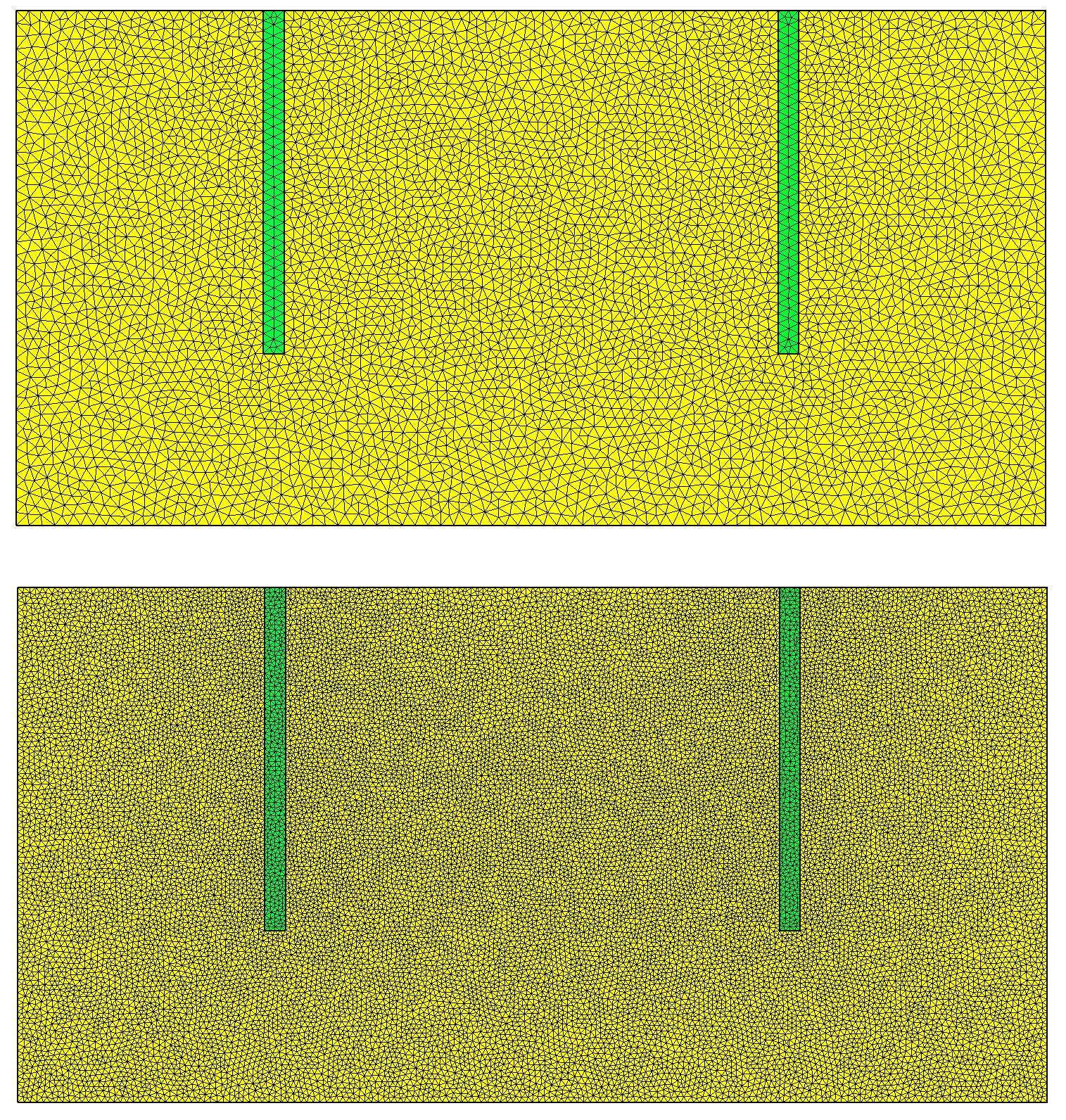}   
\end{center}
\caption{Computational grids (10000 and 40000 cells)}
\label{pic:mesh} 
\end{figure}

\begin{figure}
\begin{center}
\includegraphics[width=0.9\linewidth]{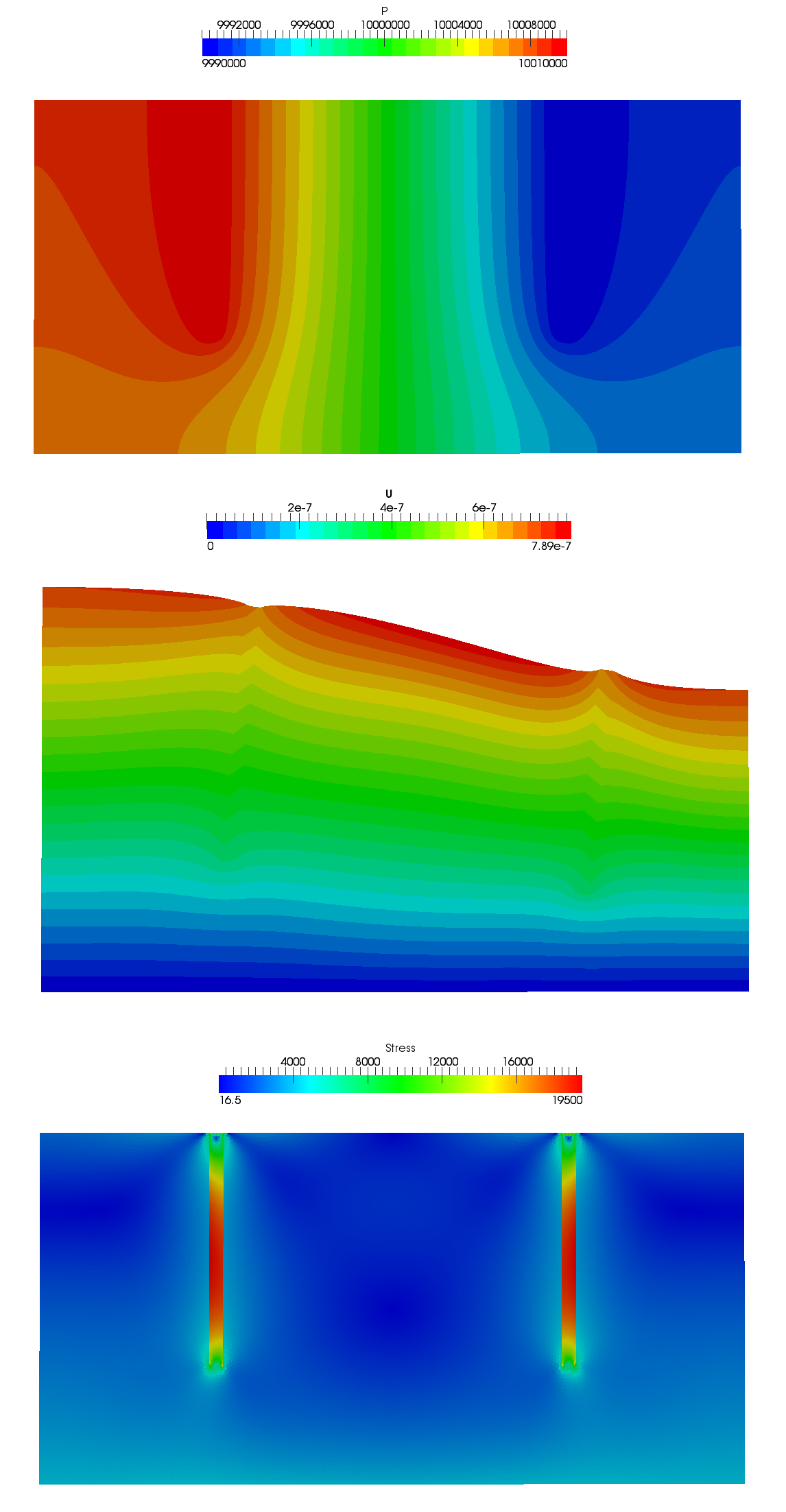}
\end{center}
\caption{Pressure (upper), displacement ($\times 10^6$) (middle) and stress (lower) distributions}
\label{pic:sol}
\end{figure}

To compare the errors produced by the above schemes, we use the fully coupled method with the finest mesh of 40000 cells and time step $\tau = 0.1$ day in order to calculate the benchmark solution for evaluating the error.  
For error comparison we use 
\[
\varepsilon_p = || p_e - p ||,
  \quad  || p_e - p ||^2 =  \int_{\Omega} ( p_e -  p)^2 d\bm x ,
\]
where $ p_e$ is the benchmark solution of pressure and $p$ -- calculated pressure.
Figures \ref{pic:err-good} and \ref{pic:err-bad} demonstrate 
the errors obtained using the splitting schemes 
(\ref{eq:sc-D}),  (\ref{eq:sc-L}), (\ref{eq:sc-U}) (D, L, U on Fig.~\ref{pic:err-good} and Fig.~\ref{pic:err-bad})   
and modifications (\ref{eq:RL}), (\ref{eq:RU}) (ML, MU on Fig.~\ref{pic:err-good} and Fig.~\ref{pic:err-bad}) on the mesh of 10000 cells for Tests 1 and 2, respectively. 

For Test 2, We observe numerical instability for our additive representations of the operator $\mathbb{B}$.
But applying the modified schemes for Test 2, we obtain the stable solution. 
As for Test 1, all schemes do work.

\begin{figure}
\begin{center}
\includegraphics[width=0.8\linewidth]{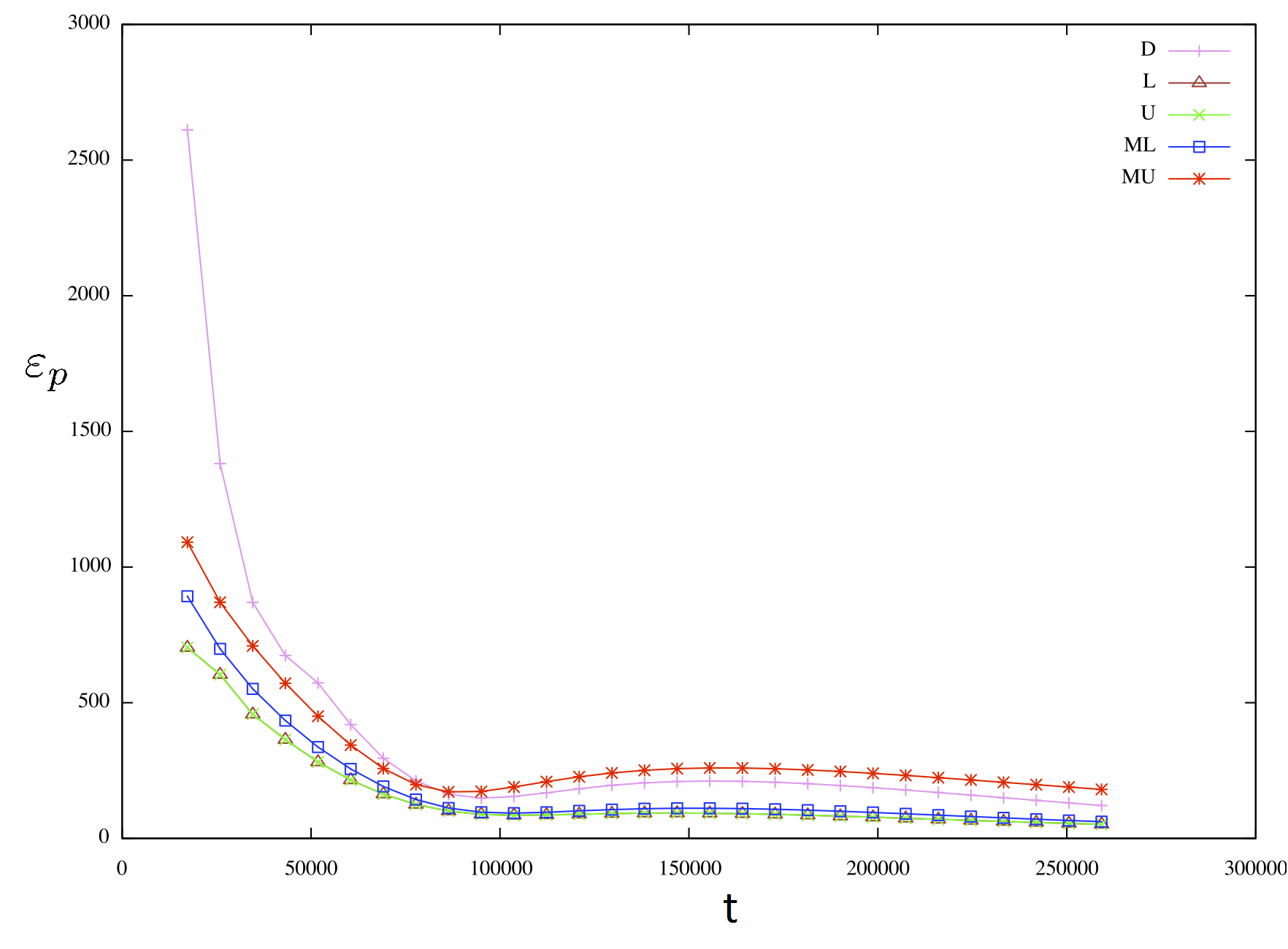} 
\end{center}
\caption{Comparison of the error for the pressure (Test 1)}
\label{pic:err-good} 
\end{figure}

\begin{figure}
\begin{center}
\includegraphics[width=0.8\linewidth]{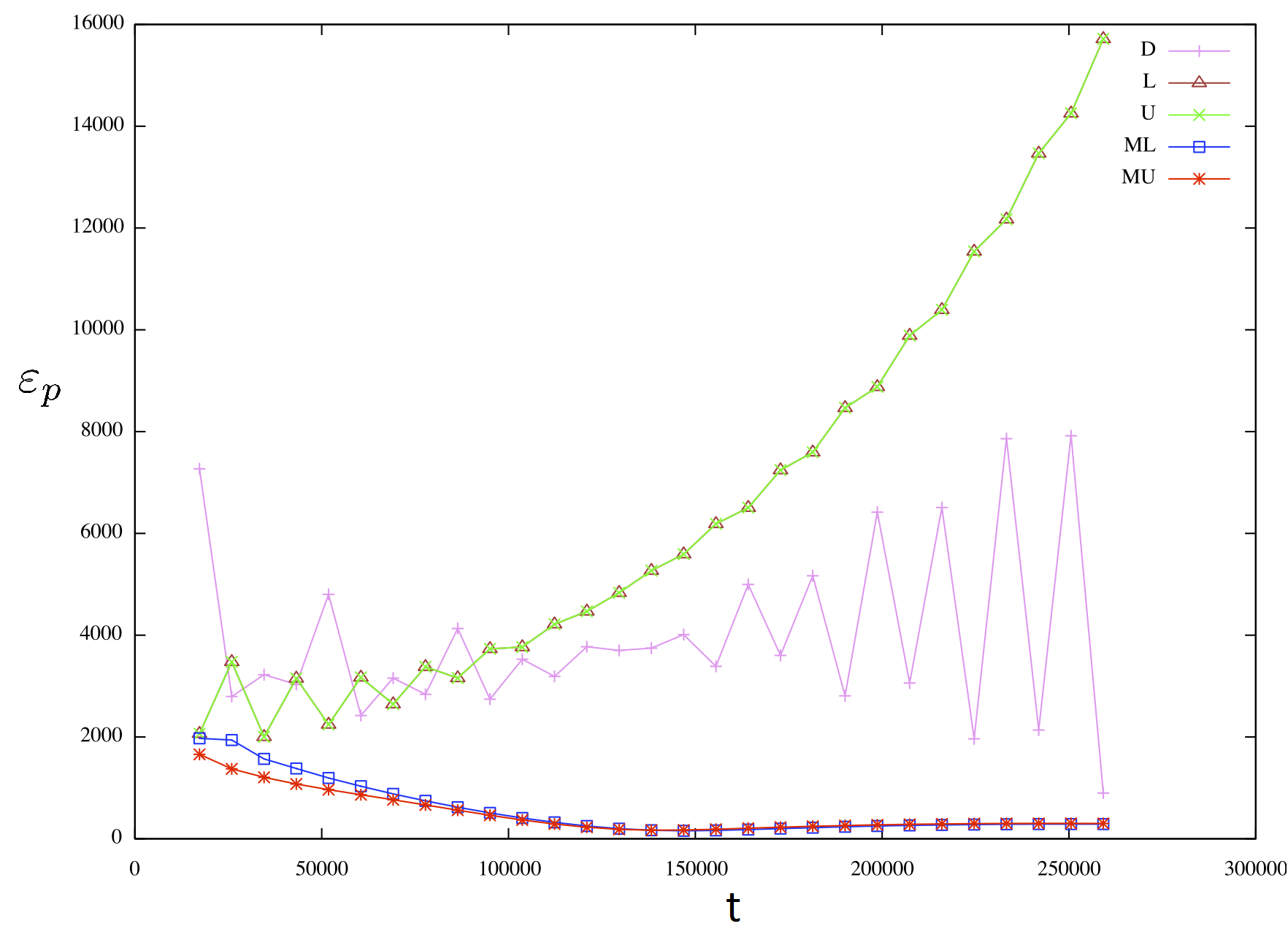} 
\end{center}
\caption{Comparison of the error for the pressure (Test 2)}
\label{pic:err-bad} 
\end{figure}

\section{Three-dimensional poroelasticity problem}

Next, we  consider a three-dimensional reservoir with one vertical injection well and four horizontal production wells.  The reservoir domain is shown in Fig.~\ref{fig:3d-domain} with plotted wells. The domain is as large as $200 \times 200  \times 20 $ m. The production and injection wells  work with  $p_p=9.9975$ MPa and $p_i=10.001$, respectively. The rock and fluid parameters of the problem are similar to Test 2 from the previous section. 

\begin{figure}
\begin{center}
\includegraphics[width=0.8\linewidth]{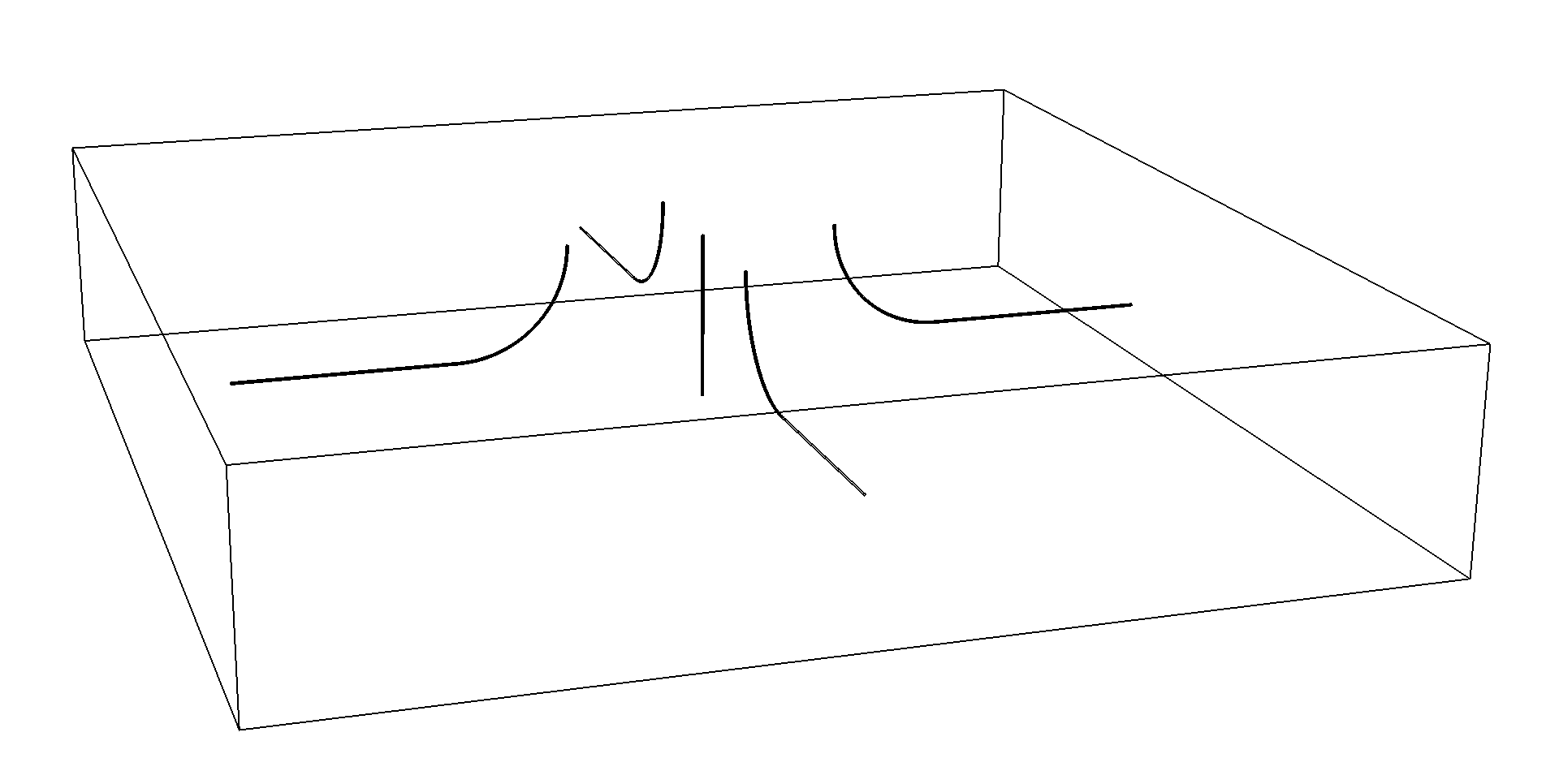}
\caption{Reservoir with plotted wells}
\label{fig:3d-domain}
\end{center}
\end{figure}

Three meshes of different quality were used for predictions (see Fig.~\ref{fig:3d-mesh} for the coarse (upper), medium (middle) and fine (lower picture) grid). 
The numbers of vertices, cells, and degrees of freedom (dof) of $p$, $\bm u$, and $\bm w$ for each mesh are given in Table~\ref{tab:meshes}.  We see that the number of dofs for the displacement field $\bm u$ is much higher than the dofs number for the pressure field $p$. This results from the fact that of the quadratic vector element and the linear scalar element  are used for discretization of the displacement and pressure, respectively. For the fully  coupled method, we search the combined vector $\bm w$ that is sum of $\bm u$ and $p$. 

\begin{figure}
\begin{center}
\includegraphics[width=0.9\linewidth]{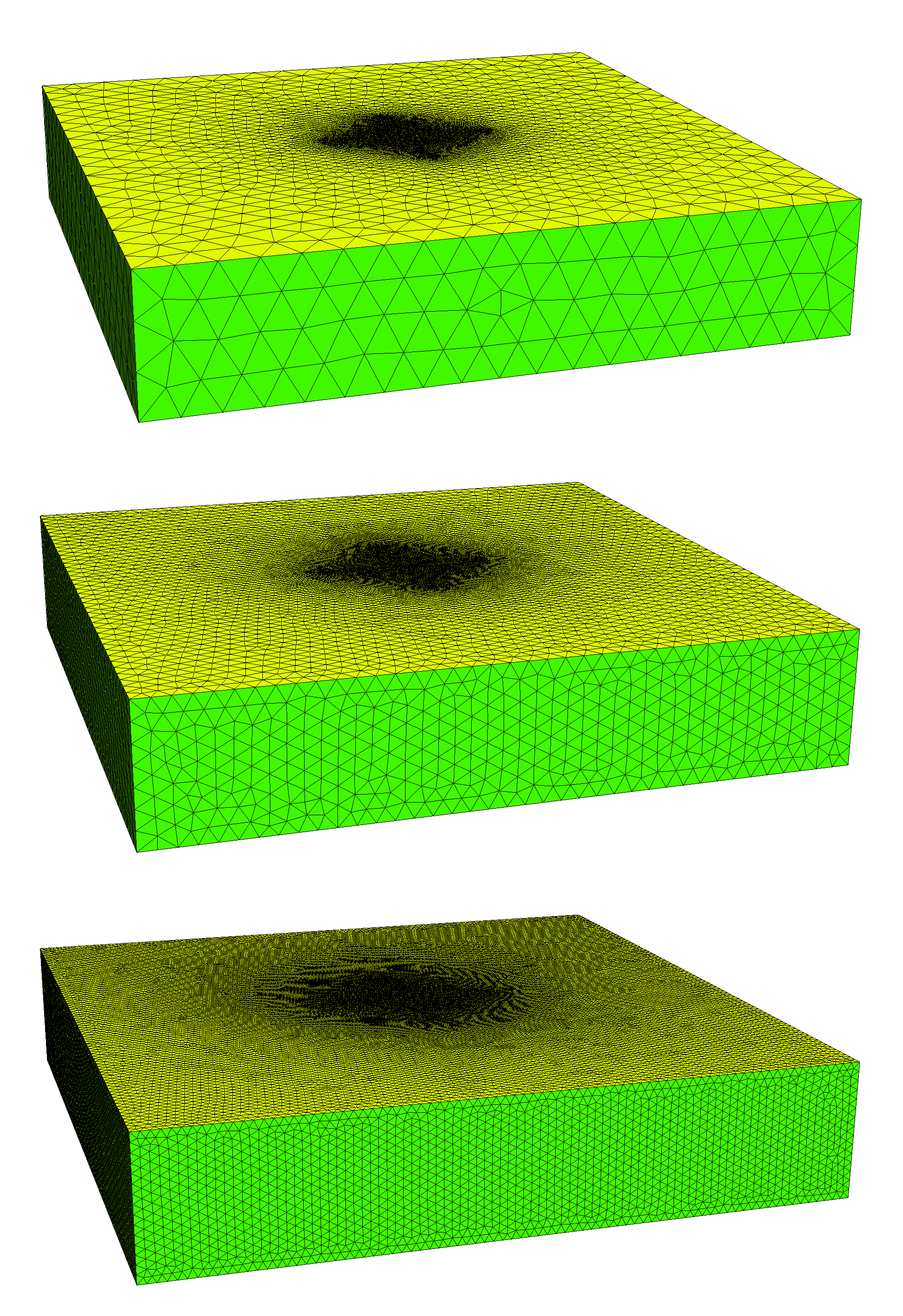}
\caption{Computational meshes: the coarse (upper), medium (middle) and fine (lower picture) grid}
\label{fig:3d-mesh}
\end{center}
\end{figure}

\begin{figure}
\begin{center}
\includegraphics[width=1.0\linewidth]{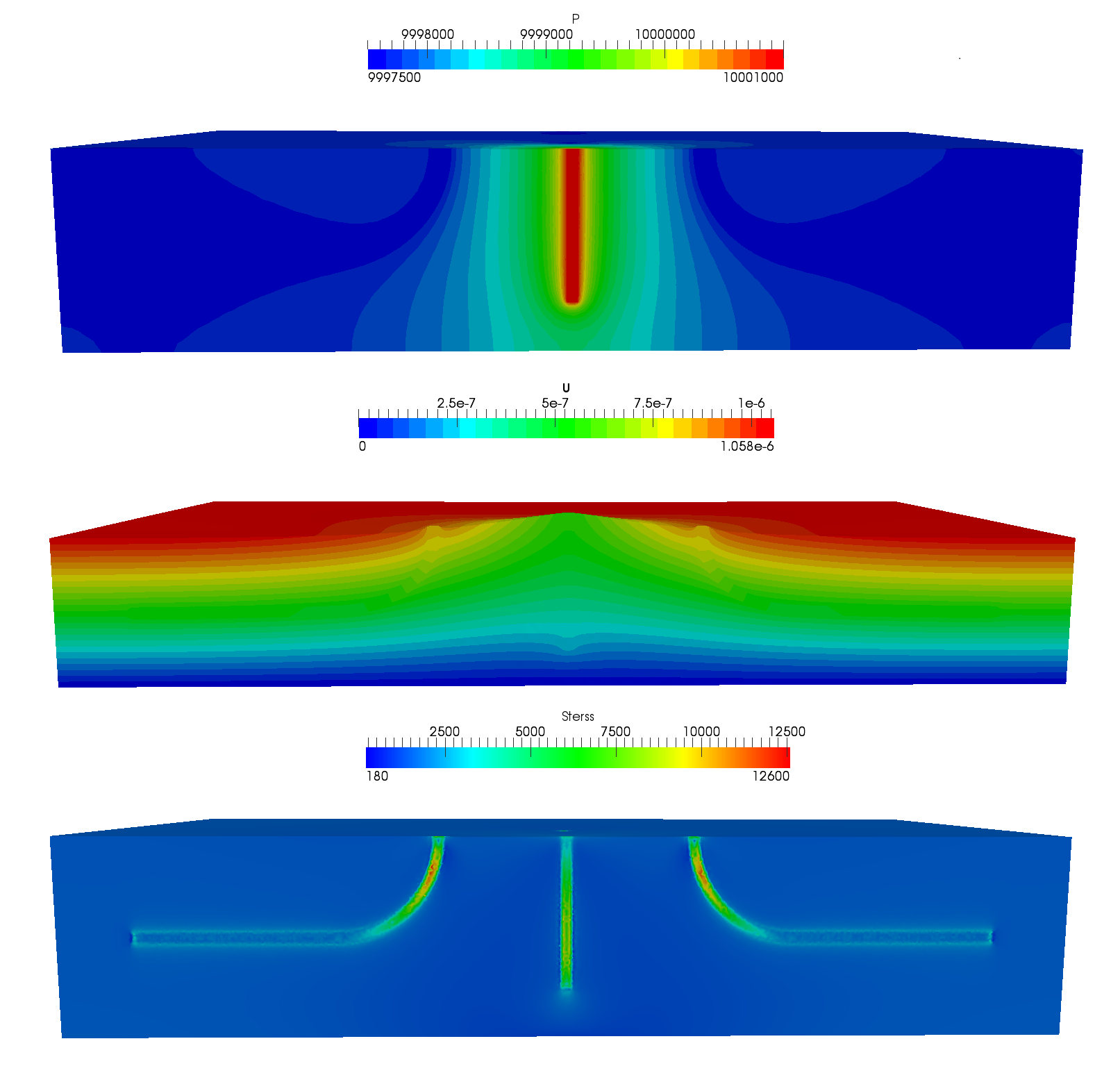}
\end{center}
\caption{Pressure (upper), displacement ($\times 5.0\cdot 10^{6}$) (middle) and stress (lower) distributions}
\label{pic:3d-sol}
\end{figure}

\begin{table}
\begin{center}
\begin{tabular}{lrrr}
\hline 
   & mesh1 		& mesh2 		& mesh3 \\ \hline
Vertices								& 19005		& 30502		& 78956 \\
Cells										& 95088 		&	146541  	& 	376056 \\
Number of $p$  dof				&	19005		& 30502		 & 78956 \\
Number of $\bm u$ dof  	& 415287		& 655179		 & 1691586 \\
Number of $\bm w$ dof 		& 434292		& 685681	 	 & 1770542 \\
\hline 
\end{tabular} 
\end{center}
\caption{Parameters of meshes}
\label{tab:meshes}
\end{table}

The results of numerical experiments are presented in Fig.~\ref{pic:3d-sol}.  The time step $\tau=1$ day and the maximum time $T_{max} = 30$ day were used.
The generalized minimal residual method (GMRES) with incomplete LU-factorization preconditioner (ILU) was chosen as the linear solver for of both schemes.  
Numerical experiments on a different number of processors were conducted.
Calculation times for the coupled  scheme and the splitting schemes (MU) are shown in Table~\ref{tab:coupled_split_comparison}.  A good parallelization efficiency is observed.

The parallel code was run on a cluster \textit{Arian Kuzmin} of North–Eastern Federal University.
The cluster consists of 160 computing nodes, each node has two 6-core processors Intel Xeon X5675 3.07 GHz with 48 GB RAM.

\begin{table}
\begin{center}
\begin{tabular}{r|rrrrrr}
\hline 
np	& \multicolumn{2}{c}{mesh1}  & \multicolumn{2}{c}{mesh2} 	&	\multicolumn{2}{c}{mesh3} 		\\	
\hhline{~------}
		& coupled & split & 	coupled	& split	& coupled	 &		split \\ \hline
1		&	4869.57	&	3129.18	&	9140.62	&	6174.50	&	40918.40		&	32566.40	\\
2		&	2101.24	&	1247.07	&	3826.16	&	2438.96	&	16813.20		&11360.30	\\
4		&	1020.60	&	572.50		&	1647.10	&	1059.33	&	5959.41		&	4285.02	\\
8		&	598.61		&	376.68		&	956.65		&	624.52		&	3105.15		&	2177.98	\\
16	&	451.77		&	322.40		&	635.43		&	492.61		&	1768.68		&	1235.90	\\
32	&	230.23		&	200.69		&	383.93		&	312.09		&	1015.04		&	657.06	\\ 
\hline 
\end{tabular} 
\end{center}
\caption{Comparison of coupled and RU-split methods run times}
\label{tab:coupled_split_comparison}
\end{table}

\section{Conclusions}

\begin{enumerate}
\item Stability estimates of weighted schemes for the coupled system of equations are obtained for the differential and discrete problem using Samarskii's theory of stability for operator-difference schemes.
\item Splitting schemes are constructed using an additive representation of the operator at the time derivative.  Undrained and fixed stress split methods are presented as regularized schemes. 
\item It was found that the additive schemes do not always work, and for some problem parameters, we need to use the regularized schemes. 
\item For solving the three-dimensional problem, parallel computations were performed using the standard technique. They demonstrate good parallelization efficiency for both the coupled and splitting schemes.

\end{enumerate}

\end{document}